\documentclass[12pt,a4paper]{article}
\usepackage[T1]{fontenc}
\usepackage{amssymb,amsfonts,amsmath,amsthm}
\usepackage[russian,english]{babel}
\usepackage{longtable}
\usepackage[cp1251]{inputenc}
\usepackage[all]{xy}
\usepackage{graphicx}
\usepackage{wrapfig}

\usepackage[colorlinks]{hyperref}
\usepackage{multirow}
\usepackage{hhline}

\usepackage{cancel}

\DeclareMathOperator{\rank}{\mathrm{rank}}

\DeclareMathOperator{\ind}{\mathrm{ind}}

\DeclareMathOperator{\Ad}{\mathrm{Ad}}

\DeclareMathOperator{\arcsinh}{arcsh}
\DeclareMathOperator{\cotg}{ctg}
\DeclareMathOperator{\tgg}{tg}

\newcommand{\ba}{\bar{\alpha}}

\newcommand{\diag}{\mathrm{diag}}
\newcommand{\R}{\mathbb{R}}
\newcommand{\CC}{\mathbb{C}}
\newcommand{\F}{\mathbb{F}}

\newcommand{\z}{\mathfrak{z}}
\newcommand{\<}{\langle}
\renewcommand{\>}{\rangle}

\newcommand{\NM}{\mathfrak{N}(M)}
\newcommand{\HH}{\mathfrak{H}}

\newcommand{\gf}{\mathfrak{g}^*}
\newcommand{\g}{\mathfrak{g}}
\newcommand{\h}{\mathfrak{h}}
\newcommand{\s}{\mathfrak{s}}
\newcommand{\kk}{\mathfrak{k}}
\newcommand{\sm}{\mathfrak{s}(M)}

\newcommand{\eee}{\mathfrak{e}}
\newcommand{\pp}{\mathfrak{p}}

\newcommand{\Ug}{{\cal{U}(\mathfrak{g})}}
\newcommand{\Zg}{\cal{Z}(\mathfrak{g})}

\newcommand{\Ol}{{\cal O}_\lambda}

\newcommand{\p}{\partial}
\newcommand{\DM}{\mathfrak{D}(M)}

\newcommand{\bal}{\bar{\alpha}}
\newcommand{\zxx}{\zeta_X}
\newcommand{\tx}{\theta_X}

\newcommand{\DD}{\Delta_2}

\newtheorem{proposition}{Proposition}
\newtheorem{theorem}{Theorem}

\newtheorem{definition}{Definition}
\newtheorem{lemma}{Lemma}

\theoremstyle{remark}
\newtheorem{remark}{Remark}

\newcommand{\pder}[2]{\frac{\partial\, #1}{\partial\, #2}}

\title{Harmonic Variables for Laplace Operators on Homogeneous Spaces}

\author{Shirokov I.V.\thanks{iv\_shirokov@mail.ru}\\
Omsk State Technical University, \\
644050, Omsk, Mira Ave., 11}

\begin{document}
\selectlanguage{english}
\maketitle

\begin{abstract}
\noindent \emph{A definition of harmonic variables for Laplace operators on Riemannian and pseudo-Riemannian spaces is introduced. Harmonic variables allow, in particular, the construction of particular (non-invariant group) solutions to a wide class of nonlinear differential equations. For Lie groups with a bi-invariant metric and a certain class of homogeneous spaces (e.g., symmetric spaces), an algorithm for constructing such variables is presented.}
\end{abstract}


\vspace{1cm}
\textbf{UDC}: 512.816.7

\section*{Introduction}

Consider the Laplace equation:
\begin{equation}\label{e1}
\left(\dfrac{\p^2 }{\p x^2}+\dfrac{\p^2 }{\p y^2}\right)u(x,y)=0.
\end{equation}

Recall that functions that are solutions to the Laplace equation are called \textit{harmonic}.

\begin{definition}
A curvilinear, generally complex-valued variable $q=q(x,y)$ will be called harmonic if for any arbitrary smooth function of this variable $v(q)$, the function $u(x,y)=v(q(x,y))$ is harmonic.
\end{definition}

\begin{remark}
One should not confuse the concept of a harmonic variable introduced here with the well-known \cite{Fok} term harmonic coordinates, where the equality
$$\frac{\p}{\p x^i} \left(\sqrt{|g|}g^{ij}(x)\right)=0$$
holds.
\end{remark}


It is well known that the function $u(x,y)=v(x+i y)$ for any analytic function of a complex argument $v(q)$ satisfies equation \eqref{e1}, i.e., the variable $q=x+i y$ is harmonic.

Another simple example is the Laplace equation in two-dimensional pseudo-Euclidean space:
\begin{equation}\label{e2}
\left(\dfrac{\p^2 }{\p t^2}-\dfrac{\p^2 }{\p x^2}\right)u(t,x)=0.
\end{equation}
Obviously, the variables $q_1=t+x$ and $q_2=t-x$ are harmonic for equation \eqref{e2}.

Let $(M,\textrm{ g})$ be a Riemannian (or pseudo-Riemannian) space with a metric $\textrm{ g}=g_{ij}(x)dx^idx^j$. Consider the Laplace-Beltrami equation on this space:
\begin{equation}
\Delta_2 u(x)=0,\quad \Delta_2\equiv \frac{1}{\sqrt{|g|}}\,
\dfrac{\p}{\p x^i} \sqrt{|g|} g^{ij}(x)\,\dfrac{\p}{\p x^j},\quad u\in C^\infty (M).
\end{equation}
For the Laplace-Beltrami equation on a Riemannian (or pseudo-Riemannian) space, one can also pose the problem of finding harmonic variables. Note that in the Riemannian case, any harmonic variable is complex.


Let us point out one application of harmonic variables. Consider a nonlinear differential equation of the form:
\begin{equation}\label{e51}
\frac{\p^2 u(t,x)}{\p t^2}-\Delta_2 u(t,x)=F(t,u(t,x)).
\end{equation}
If $q=q(x)$ is a harmonic variable, and $w=w(t,C_1,C_2)$ is the general solution of the ordinary differential equation:
\[
\frac{d^2 w}{d t^2}=F(t,w),
\]
where $C_i$ are arbitrary constants, then we can write down solutions to equation \eqref{e51} that contain arbitrary functions of the harmonic variable:
$$u(t,x)=w(t,C_1(q(x)),C_2(q(x)).$$
Note that such solutions, in general, are not invariant with respect to any symmetry group.

To distinguish Laplace operators on different spaces $M$, we will denote them by the symbol $\Delta_2(M)$. For example, the operator of equation \eqref{e1} is denoted by $\Delta_2(R^2)$, and that of equation \eqref{e2} by $\Delta_2(R^{1,1})$, etc.

As a non-trivial example, consider the Laplace-Beltrami equation on the two-dimensional sphere $S^2$:
\begin{equation*}
\Delta_2(S^2)u=\left( \frac{1}{\sin \theta}\frac{\p}{\p \theta}\sin \theta \frac{\p}{\p \theta}+
 \frac{1}{\sin^2\theta}\frac{\p^2}{\p \varphi^2}\right)u(\varphi,\theta)=0.
 \end{equation*}
By direct substitution, one can verify that the variable $q(\varphi,\theta)$ is harmonic:
\begin{equation*}
q(\varphi,\theta)=\varphi -  i\arcsinh ( \cotg \theta ).
\end{equation*}

As an application of this result, consider the nonlinear differential equation:
\[
\left(\pder{{}^2}{x^2}+\pder{{}^2}{y^2}+\pder{{}^2}{z^2}\right)u(x,y,z)=F(u(x,y,z)),
\]
which we write in the spherical coordinate system:
\begin{equation}\label{e52}
\left(\frac{1}{r^2}\pder{}{r}r^2\pder{}{r}+\frac{1}{r^2}\DD(S^2)\right)
u(r,\varphi,\theta)=F(u(r,\varphi,\theta)).
\end{equation}
Let $w=w(r,C_1,C_2)$ be the general solution of the ordinary differential equation
\[
\frac{1}{r^2}\frac{d}{d r}r^2\frac{d}{d r}w=F(w),
\]
then $u(r,\varphi,\theta)=w(r,C_1(q(\varphi,\theta)),C_2(q(\varphi,\theta)))$ is a particular solution of equation \eqref{e52}, containing two arbitrary functions of the combination of angular variables.

It is easy to show that the two-dimensional Laplace-Beltrami operator always admits a harmonic variable.

Note that for an arbitrary function $v(q)$ of a harmonic variable $q(x)$ such that the function $u(x)=v(q(x))$ is analytic in a domain $D\subset R^n$, the following equality holds:
\begin{equation}\label{e6a}
\oint\limits_{\p D}g^{ij}(x)\pder{v(q(x))}{x^j}\,dS_i=0,\quad dS_i=\sqrt{|g|}\, dx^1\wedge \dots \wedge dx^{i-1}\wedge \cancel{dx^i}\wedge dx^{i+1}\wedge \dots \wedge dx^n.
\end{equation}
For the two-dimensional Laplace operator \eqref{e1} and the harmonic variable $q=x^1+i x^2$, this equation represents Cauchy's theorem in the theory of functions of a complex variable. In this sense, harmonic variables are a generalization of complex variables. In this paper, we will not address issues related to equality \eqref{e6a} and other possible applications.

Let an arbitrary function $v$ of the variable $q(x)$ be harmonic. Then we have the equality:
\begin{equation}\label{e3}
\Delta_2 v(q(x))=v'(q)\Delta_2 q(x)+v''(q)\Delta_1[q]=0
\end{equation}
Here
\[
\Delta_1[q]\equiv g^{ij}(x)\frac{\p q(x)}{\p x^i}\frac{\p q(x)}{\p x^j}.
\]
Due to the arbitrariness of the function $v$, equality \eqref{e3} splits into two equations for the single unknown function $q(x)$:
\begin{eqnarray}\label{e4a}
\Delta_2\, q(x)=0;\\ \label{e4b}
\Delta_1[q(x)]=0.
\end{eqnarray}
Equation \eqref{e4a} shows that a harmonic variable is a harmonic function (the converse is not true). Equation \eqref{e4b} means that the harmonic variable must also satisfy the eikonal equation, which describes the propagation of electromagnetic waves in a Riemannian space in the geometric optics approximation.

Let us denote by $Symm(\Delta_{1,2}(M))$ the symmetry group of the system of equations \eqref{e4a}, \eqref{e4b}. This group contains as subgroups: the group of motions of the Riemannian space $(M,\textrm{g})$, the group of conformal transformations of this space, as well as the infinite group of transformations $q(x)\to \Phi(q(x))$. The latter means that any arbitrary function $\Phi(q(x))$ of a harmonic variable $q(x)$ is itself a harmonic variable.
Two harmonic variables $q_1(x),\; q_2(x)$ lying on the same $Symm(\Delta_{1,2}(M))$--orbit will be called \textit{equivalent}.

Consider the simplest example. The harmonic variable $q=x+iz$ for the three-dimensional flat Laplace operator $\Delta_2(R^3)=\p_x^2+\p_y^2+\p_z^2$, after a rotation in the $X-Y$ plane, transforms into the equivalent harmonic variable $\tilde{q}=x\cos\alpha+y\sin\alpha+i z$.



\begin{definition}
We will say that a set of functionally independent harmonic variables $q_1(x)$, $q_2(x)$, $\dots,q_s(x)$ forms a complex if any arbitrary smooth function $v$ of these variables $u(x)=v(q_1(x),q_2(x),\dots,q_s(x))$ is harmonic.
\end{definition}


The Laplace operator \eqref{e1} admits two harmonic variables: $q_1=x+iy,\ q_2=x-iy$. The Laplace operator \eqref{e2} also admits two harmonic variables: $q_1=t+x,\ q_2=t-x$. Obviously, in neither of these cases do the two harmonic variables form a complex. It is easy to provide an example of a complex: for the operator $\DD (R^4)$, the harmonic variables $q_1=x_1+i x_2,\ q_2=x_3+i x_4$ form a complex. Of course, such trivial complexes are of little interest to us.

The main goal of this paper is to describe classes of homogeneous spaces and metrics for which harmonic variables exist, and to present an algorithm for finding them.


\section{Some Preliminaries and Notation}

In this section, for the sake of completeness, we will present well-known facts from the theory of Lie groups and homogeneous spaces and fix the notation. Note immediately that all constructions in this paper are of a local nature.

\subsection{Lie Groups}

Let $G$ be a real $n$-dimensional Lie group, and $\g$ its Lie algebra. The group $G$ acts on itself by right $R$ and left $L$ translations; their differentials define left- and right-invariant vector fields $\xi_X(g),\; \eta_X(g) \in T_gG$:
\begin{equation*}
L_{g_1}g=g_1\,g,\quad \left(L_{g_1}\right)_* X=\xi_X(g_1);\quad
R_{g_1}g=g\,g_1,\quad \left(R_{g_1}\right)_* X=-\eta_X(g_1); \quad X\in \g\approx T_eG.
\end{equation*}

In the space of smooth functions $C^\infty(G)$ on the group $G$, left and right translations define the left $T^L$ and right $T^R$ regular representations:
\begin{equation}\label{e7}
T^L_{g_1} u(g)=u(g^{-1}_1g), \quad T^R_{g_1} u(g)=u(gg_1),\quad u(g)\in C^\infty(G).
\end{equation}

Choosing $g_1=\exp(t\,X),\ X\in\g$ in formula \eqref{e7} and differentiating these equalities with respect to the parameter $t$ at zero, we obtain representations $\tau^L,\;\tau^R$ of the algebra $\g$ by differential operators
$\xi_X,\ \eta_X$ in the space of functions $C^\infty(G)$:
\begin{eqnarray*}
\tau^R(X)u(g)&=&\left.\left(\frac{d}{dt}T^R_{\exp (t X)}u(g)\right)\right|_{t=0}=\xi_X(g)u(g),\\
\tau^L(X)u(g)&=&\left.\left(\frac{d}{dt}T^L_{\exp (t X)}u(g)\right)\right|_{t=0}=\eta_X(g)u(g),\quad
X\in\g.
\end{eqnarray*}
\begin{equation*}
[\xi_X,\xi_Y]=\xi_{[X,Y]},\quad [\eta_X,\eta_Y]=\eta_{[X,Y]},\quad [\xi_X,\eta_Y]=0,\quad X,Y\in\g.
\end{equation*}

We also recall the adjoint representation of the group $\Ad_g=(L_g)_*(R_{g^{-1}})_*: \g\to\g$.
\begin{equation*}
\Ad_{g_1}\Ad_{g_2}=\Ad_{g_1g_2},\quad \xi_i(g)=- (\Ad_g)_i^j\eta_j(g).
\end{equation*}
Here and below, $\{E_i\}$ is a basis of the algebra $\g$, $\xi_i\equiv\xi_{E_i},\ \eta_i\equiv\eta_{E_i}$.

Let us choose local coordinates $x=(x^1,\dots,x^n)$ in a neighborhood of the identity $U_e\subset G$ of the group $G$. We will denote by $g_x$ the group element with coordinates $x$. Canonical coordinates of the first kind $g_x=\exp\left(\sum_i x^i E_i\right)$ or second kind $g_x=\prod_{i}\exp(x^i E_i)$ are often used.
In local coordinates, the product in the group is defined by the composition function $\varphi(x,y)$:
\begin{equation*}
g_x g_y=g_z,\quad z^i=\varphi^i(x,y),\quad g_x,g_y,g_z\in U_e.
\end{equation*}
Without loss of generality, we assume that the zero value of the coordinates corresponds to the identity element: $g_0=e$. Let $\chi(x)$ denote the coordinates of the inverse element: $g_{\chi(x)}=g_x^{-1}$.
For the subsequent presentation, we will need expressions for invariant fields, forms, and measures in coordinate form:
\begin{gather}\label{e8}
\xi_k(g_x)=\xi_k^a(x)\frac{\p}{\p x^a},\quad \xi_k^a(x)=\left.\frac{\p \varphi^a(x,y)}{\p y^k}\right|_{y=0};\\ \nonumber
\eta_k(g_x)=\eta_k^a(x)\frac{\p}{\p x^a},\quad \eta_k^a(x)=\left.\frac{\p \varphi^a(\chi(y),x)}{\p y^k}\right|_{y=0}=\left.-\frac{\p \varphi^a(y,x)}{\p y^k}\right|_{y=0};\\ \nonumber
\omega^i(g_x)=\omega^i_k(x)\,d x^k,\ \omega^i_k(x)=(\|\xi(x)\|^{-1})^i_k;\quad
\sigma^i(g_x)=\sigma^i_k(x)\,d x^k,\ \sigma^i_k(x)=(\|\eta(x)\|^{-1})^i_k;\\
d\mu^L(g_x)=|\det \omega^i_j(x)|\,dx,\quad d\mu^R(g_x)=|\det \eta^i_j(x)|\,dx,\quad
dx\equiv dx^1\wedge\dots \wedge dx^n. \nonumber
\end{gather}
Here $E_k=\p/\p x^k|_{0}\in T_eG\approx \g;\ \omega^i,\;\sigma^i$ are left- and right-invariant forms, $d\mu^L,\;d\mu^R$ are left- and right-invariant measures:
\begin{equation*}
d\mu^L(g_1g)=d\mu^L(g),\quad d\mu^R(g g_1)=d\mu^R(g),\quad d\mu^L(g)=\Delta(g) d\mu^R(g),\quad \Delta(g)=|\det (\Ad_{g^{-1}})_i^j|.
\end{equation*}
Recall that a group is called \textit{unimodular} if $\Delta(g)=1$.

Let us recall a few more terms that will be useful.
Consider the space of smooth functions $C^\infty(\gf)$ of formal variables $f_1,\dots,f_n$ ($n=\dim\g$), which we will consider as coordinates of a generic linear functional $f\in\gf,\ f_i=\<f,E_i\>$. The group $G$ acts on the linear space $\gf$ by the coadjoint representation: $f\to \Ad_g^* f$ according to the rule: $\<\Ad_g^* f,X\>=\<f,\Ad_g X\>$, or in coordinates $(\Ad_g^* f)_i=(\Ad_g)_i^j f_j$. Functions $K(f)\in C^\infty(\gf)$ invariant under the coadjoint representation are called \textit{Casimir functions}.

The \textit{index} $\ind\g$ of an algebra $\g$ is the non-negative integer
\[
\ind\g=\inf_{\lambda\in \gf}\dim \g^\lambda,\quad \g^\lambda=
\{X\in\g\mid \<\lambda,[X,\g]\>=0\}.
\]
The index of an algebra coincides with the number of independent Casimir functions.
Note that the rank of a semi-simple algebra coincides with its index.

A \textit{polarization} \cite{Diksime} of an algebra $\g$ with respect to a linear functional $\lambda\in\gf$ is a subalgebra $\kk\subset\g^c$ such that
\begin{equation} \label{e9}
\<\lambda,[\kk,\kk]\>=0,\quad \dim\kk=(\dim\g+\dim\g^\lambda)/2.
\end{equation}

\subsection{Homogeneous Spaces}

Let $M$ be a homogeneous $m$-dimensional $G$-space, equivalent to the space of right cosets $M=G/H$, where $H$ is a closed subgroup of the group $G$, being the stationary subgroup of some distinguished point $x_0\in M$. The action of the group $G$ on the space $M$: $x\to xg$ generates a regular representation in the space of smooth functions $C^\infty(M)$:
\begin{equation}\label{e14}
T^M_gu(x)=u(xg),\quad T^M_{g_a}T^M_{g_b}=T^M_{g_ag_b}.
\end{equation}

In local coordinates, the action of the group is defined by the function $(xg_a)^\alpha=\psi^\alpha(x,a)$, and equality \eqref{e14} takes the form:
\begin{equation}\label{e15}
\psi^\alpha(\psi(x,a),b)=\psi^\alpha(x,\varphi(a,b)).
\end{equation}

Differentiating the regular representation \eqref{e14} at the identity of the group, we obtain a representation $\tau^M$ of the algebra $\g$ by vector fields on the manifold $M$:
\begin{multline}\label{e16}
\tau^M(X) u(x)=\left.\left(\frac{d}{dt}T^M_{\exp (t X)}u(x)\right)\right|_{t=0}=\zxx(x)u(x),
\quad \zxx(x)=\zxx^\alpha(x)\frac{\p}{\p x^\alpha},\\
 \zeta_{E_k}^\alpha(x)=\left.\frac{\p \psi^\alpha(x,a)}{\p a^k}\right|_{a=0}.
\end{multline}
In \cite{MMS}, a method is presented for constructing the fields $\zxx$ in local coordinates using the structure constants of the algebra $\g$ and the subalgebra $\h$.

Let us give an important formula for the subsequent exposition:
\begin{equation}\label{e16a}
{(T_g^M)}_{*}\zeta_i(x)=(\Ad_{g^{-1}})_i^k\zeta_k(xg)\longleftrightarrow
\zeta_i^\beta(x)\frac{\p (x g)^\alpha}{\p x^\beta}=
(\Ad_{g^{-1}})_i^k\zeta_k^\alpha(xg).
\end{equation}

Consider another homogeneous $G$-space $Q$, equivalent to the space of right cosets $Q=G/K$. The action of the group $G$ generates the corresponding regular representation:
\begin{equation*}
T^Q_gu(q)=u(qg),\quad g\in G;\quad q,qg\in Q,\quad u(q)\in C^\infty(Q),
\end{equation*}
and the representation $\tau^Q$ of the algebra $\g$ by fundamental vector fields $\tx$ on $Q$:
\begin{equation*}
\tau^Q(X)u(q)=\left.\left(\frac{d}{dt}T^Q_{\exp (t X)}u(q)\right)\right|_{t=0}=\tx(q)u(q),
\quad \tx(q) = \tx^{\bal} (q)  \frac{\p}{\p q^{\bal}},\quad X\in\g.
\end{equation*}

Recall the concept of a morphism of homogeneous spaces. Let $\dim M\geq\dim Q$. A mapping $I:\ M\to Q$ is called a \textit{morphism} if the following equality holds:
\begin{equation}\label{e19}
I^*\; T_g^Q=T_g^M\; I^*
\end{equation}
Here the mapping $I^*:\ C^\infty(Q)\to C^\infty(M)$ acts according to the rule: for an arbitrary function $v(q)\in C^\infty(Q)$, the function $I^*(v)(\cdot)=v(I(\cdot))\in C^\infty(M)$. In other words, equality \eqref{e19} means:
\[
v(I(x)g)=v(I(xg)),\quad \forall v(q)\in C^\infty(Q).
\]
Differentiating both sides of formula \eqref{e19} with respect to the parameters of the group element, we obtain:
\begin{equation}\label{e20}
dI(\zxx(x))=\tx(q),\quad q=I(x),\quad X\in\g.
\end{equation}
The last equality in coordinates has the form:
\begin{equation}\label{e21}
\zxx^\beta(x)\frac{\p q^{\bal}}{\p x^\beta}=\tx^{\bal}(q),\quad q=I(x),\quad X\in\g.
\end{equation}

Recall \cite{Onishcik} the following statement.
\begin{proposition}
A morphism $I:\ M\to Q,\ M=G/H,\; Q=G/K$ exists if and only if there exists an element $g\in G$ such that $gHg^{-1}\subset K$.
\end{proposition}

\subsection{Algebra of Invariant Differential Operators}

The representation $\tau^M$ of the algebra $\g$ naturally extends to a representation of the universal enveloping algebra
\begin{gather*}
\Ug=\R\oplus(\g)\oplus(\g \otimes \g)\oplus\dots \oplus (\g\otimes\dots \g) /J,
\end{gather*}
(here $J$ is the two-sided ideal generated by elements $X\otimes Y-Y\otimes X-[X,Y]$) as follows: $\tau^M(X\otimes Y)=\tau^M(X)\tau^M(Y)=\zxx \zeta_Y$. Moreover,
\[
\tau^M(X\otimes Y-Y\otimes X-[X,Y])=\zxx\zeta_Y-\zeta_Y\zxx-\zeta_{[X,Y]}=0 \Longleftrightarrow \tau^M (J)=0.
\]

The adjoint representation of the group $G$ on the Lie algebra $\g$: $X\to \Ad_g\,X$ extends to the enveloping algebra $\Ug$ by the rule: $X\otimes Y\to \Ad_g\,X\otimes \Ad_g\,Y$.

The images of the representation $\tau^M$ of the center $\Zg$ of the algebra $\Ug$ are called \textit{Casimir operators} on the homogeneous space $M$.

Let $\DM$ denote the \textit{algebra of invariant differential operators} on the homogeneous space $M$, i.e., differential operators $L(x,\p_x)$ commuting with all operators $\zxx$: $[L(x,\p_x),\zxx(x,\p_x)]=0$.

Obviously, any Casimir operator ${\Lambda}(\zeta)\in \tau^M\left(\Zg\right)$ belongs to the algebra $\DM$. For a wide class of homogeneous spaces, there exist invariant differential operators not originating from the representation of the enveloping algebra. In this paper, we will be interested in the class of homogeneous spaces for which the algebra of invariant operators is exhausted by Casimir operators, i.e., $\DM=\tau^M\left(\Zg\right)$.
Such homogeneous spaces are called \textit{commutative} \cite{Vinberg} or, according to the terminology of \cite{Shir-TMF}, spaces of \textit{zero defect} ($d(M)=0$). The condition for the space $M=G/H$ to belong to this class is as follows.
\begin{equation}\label{e22}
d(M)=\frac{1}{2}\sup_{\lambda\in \h^{\bot}}\rank\<\lambda,[\g,\g]\>-
\sup_{\lambda\in \h^{\bot}}\rank\<\lambda,[\g,\h]\>=0.
\end{equation}
Here
\[
\h^{\bot}=\{f\in\gf\mid \<f,\h\>=0\}.
\]

Note that all symmetric spaces satisfy condition \eqref{e22}, i.e., are commutative.

\begin{remark}
Strictly speaking, condition \eqref{e22} is a condition for \textit{weak commutativity} (the algebra of invariant functions on the cotangent bundle is generated by Casimir functions). However, it can be shown that if condition \eqref{e22} holds, every invariant differential operator is a Casimir operator. In other words, if $A_k\in \DM$ is an invariant differential operator of order $k$, then there exists an element $\Lambda_k$ of the center of the enveloping algebra of order $k$ such that $A_k=\tau^M(\Lambda_k)$.
\end{remark}


\section{Identities in the Enveloping Algebra}

\begin{definition}
An element of the kernel of the representation $\tau^M$ acting in the enveloping algebra $\Ug$ is called an identity on the homogeneous space $M$.
\end{definition}
In other words, an element of the enveloping algebra $w=W^{i_1\dots i_k}E_{i_1}\dots E_{i_k}$ is an identity if the differential operator $\tau^M(w)=W^{i_1\dots i_k}\zeta_{i_1}(x,\p_x)\dots \zeta_{i_k}(x,\p_x)=0$ (identically equals zero). Obviously, identities form a two-sided ideal in the algebra $\Ug$, which we will denote by $\NM=\ker(\tau^M)$. The number of elements generating this ideal is called the \textit{index} of the homogeneous space and is denoted by $i_M$.

\emph{Example 1. } $\blacktriangleleft$ As a simple example, consider the three-dimensional Euclidean space $R^3\approx E(3)/SO(3)$. Here $E(3)$ is the group of motions of Euclidean space, consisting of rotations and parallel translations. The corresponding Lie algebra $\eee(3)=\{E_{ij},E_i\}$ is realized by the fundamental vector fields:
\[
\zeta_{12}=x\frac{\p}{\p y}-y\frac{\p}{\p x},\quad
\zeta_{13}=x\frac{\p}{\p z}-z\frac{\p}{\p x},\quad
\zeta_{23}=y\frac{\p}{\p z}-z\frac{\p}{\p y},\quad \zeta_1=\frac{\p}{\p x},
\quad \zeta_2=\frac{\p}{\p y}, \quad \zeta_3=\frac{\p}{\p z},
\]
Here $\tau^{R^3}(E_{ij})=\zeta_{ij},\ \tau^{R^3}(E_{i})=\zeta_{i},\ i,j=1,2,3$.
The center of the enveloping algebra $\Zg$ is generated by two elements:
$$
{\Lambda}_1=E_1^2+E_2^2+E_3^2,\quad {\Lambda}_2=E_{12}E_3-E_{13}E_2+E_{23}E_1.
$$
At the same time:
\[
\tau^{R^3}({\Lambda}_1)=\Delta_2(R^3)=\frac{\p^2}{\p x^2}+\frac{\p^2}{\p y^2}+\frac{\p^2}{\p z^2},\quad \tau^{R^3}({\Lambda}_2)=0.
\]
Thus, the element of the enveloping algebra ${\Lambda}_2$ is an identity on the three-dimensional Euclidean space (and there are no other generators, i.e., $i_{R^3}=1$).
$\blacktriangleright$

Let us represent the linear space $\g$ as a direct sum of linear subspaces (not subalgebras) $\g=\h+\pp$ and let $\{E_{\hat{\alpha}}\}$ be a basis of the subalgebra $\h$, where $\hat{\alpha}=1,\dots,\dim\h$, and $\{E_{\bar{\alpha}}\}$ ($\bar{\alpha}=\dim\h+1,\dim\h+2,\dots,\dim\g$) be a basis of the complementary space $\pp$. Below, in the algebra $\Ug$, we will use a basis consisting of monomials in which the basis elements of the algebra $\g$ are ordered in ascending order of their indices (lexicographic ordering), i.e., a general element of the enveloping algebra has the form:
\begin{equation*}
A=\sum_{k\geq 0}\sum_{i_1\leq i_2\leq \dots \leq i_k} A^{i_1\dots i_k}E_{i_1}\dots E_{i_k}\in\Ug.
\end{equation*}

Let us introduce the right ideal $\HH=\h\, \Ug$, which is generated by the subalgebra $\h$ and consists of elements of the form
$$
\HH=\{\sum_i  Y_i\otimes P_i\in\Ug\mid {Y}_{i} \in\h, \quad P_i\in\Ug\}.
$$




\begin{theorem}
Every identity belongs to an $\Ad_g$--invariant subspace of the space $\HH$. The converse is also true: every element of a $G$--invariant subspace of the ideal $\HH$ is an identity. In particular, any element of the algebra $\HH\cap\Zg$ is an identity, i.e., $\HH\cap\Zg\subset\NM$.
\end{theorem}

\emph{Proof}. $\blacktriangleleft$ Let for the homogeneous space $M=G/H$, the subgroup $H$ be the stationary subgroup of a distinguished point $x_0\in M$. This means that
\begin{equation}\label{e23}
\zeta_Y(x_0)=0,\quad Y\in\h.
\end{equation}

Let $V$ denote an $\Ad_g$--invariant subspace of the ideal $\HH$, and let $\{\Gamma_A\}$ be its basis. Let us show that every element of this subspace is an identity.

Consider the basis elements $\Gamma_A=\Gamma_A^{i_1i_2\dots i_s}E_{i_1}E_{i_2}\dots E_{i_s}$ ($i_1\leq i_2\leq \dots i_s$) of the space $V$. Due to the condition $\Gamma_A\in \HH$, we have: $\Gamma_A^{\bar{\alpha}i_2\dots i_s}=0$, i.e., $\Gamma_A=\Gamma_A^{\hat{\alpha}i_2\dots i_s}E_{\hat{\alpha}}E_{i_2}\dots E_{i_s}\ (E_{\hat{\alpha}}\in\h)$. By assumption, the space $V$ is a representation space of the group $G$, which implies the existence of a matrix $T_A^B(g)$ such that
\begin{equation}\label{e24}
\Gamma_A^{i_1i_2\dots i_s}(\Ad_g)_{i_1}^{j_1}\dots (\Ad_g)_{i_s}^{j_s}=T_A^B(g)
\Gamma_B^{j_1i_2\dots j_s}
\end{equation}
Set $x=x_0$ in formula \eqref{e16a}:
\begin{equation*}
\zeta_i^\beta(x_0)\frac{\p (x_0 g)^\alpha}{\p x_0^\beta}=
(\Ad_{g^{-1}})_i^k\zeta_k^\alpha(x_0g)
\end{equation*}
In the last formula, we introduce the notation $x=x_0 g$ and, multiplying both sides of the equality by the matrix $\Ad_g$, we finally obtain:
\begin{equation}\label{e25}
\zeta_i^\alpha(x)\pder{}{x^\alpha}=(\Ad_g)_i^j\zeta_j^\beta(x_0)\left(\pder{x^\alpha}{x_0^\beta}\right)\pder{}{x^\alpha}
\end{equation}

From formulas \eqref{e25}, \eqref{e24} we have:
\begin{gather*}
\tau^M(\Gamma_A)=\Gamma_A^{i_1\dots i_s}\;\zeta_{i_1}^{\alpha_1}(x)\pder{}{x^{\alpha_1}}\dots \zeta_{i_s}^{\alpha_s}(x)\pder{}{x^{\alpha_s}}=
\Gamma_A^{i_1i_2\dots i_s}(\Ad_g)_{i_1}^{j_1}\dots (\Ad_g)_{i_s}^{j_s}\times \\
\times \zeta_{j_1}^{\beta_1}(x_0)\left(\pder{x^{\alpha_1}}{x_0^{\beta_1}}\right)\pder{}{x^{\alpha_1}}\dots
\zeta_{j_s}^{\beta_s}(x_0)\left(\pder{x^{\alpha_s}}{x_0^{\beta_s}}\right)\pder{}{x^{\alpha_s}} = T_A^B(g)\Gamma_B^{i_1 i_2\dots i_s}\times \\
\times \zeta_{i_1}^{\beta_1}(x_0)\left(\pder{x^{\alpha_1}}{x_0^{\beta_1}}\right)\pder{}{x^{\alpha_1}}\dots
\zeta_{i_s}^{\beta_s}(x_0)\left(\pder{x^{\alpha_s}}{x_0^{\beta_s}}\right)\pder{}{x^{\alpha_s}}=T_A^B(g)\Gamma_B^{\hat{\alpha} i_2\dots j_s}\times \\
\times \zeta_{\hat{\alpha}}^{\beta_1}(x_0)\left(\pder{x^{\alpha_1}}{x_0^{\beta_1}}\right)\pder{}{x^{\alpha_1}}\;
\zeta_{i_2}^{\beta_2}(x_0)\left(\pder{x^{\alpha_2}}{x_0^{\beta_2}}\right)\pder{}{x^{\alpha_2}} \dots
\zeta_{i_s}^{\beta_s}(x_0)\left(\pder{x^{\alpha_s}}{x_0^{\beta_s}}\right)\pder{}{x^{\alpha_s}}=0.
\end{gather*}
The equality to zero at the end of this chain of formulas follows from formula \eqref{e23}: $\zeta_{\hat{\alpha}}^{\beta_1}(x_0)=0$.

Now let $\Gamma$ be an identity; let us show that this element of the enveloping algebra belongs to an invariant subspace of the ideal $\HH$.

Let $(\Ug)_{k\geq 0}$ be the canonical filtration of the algebra $\Ug$, and let $\Gamma\in {\cal{U}}_k(\g)\cap \NM$ for some $k$. Obviously, for any $X\in\g$:
$$
\Gamma_1=[X,\Gamma]\in {\cal{U}}_k(\g)\cap \NM,\quad
\Gamma_2=[X,\Gamma_1]\in {\cal{U}}_k(\g)\cap \NM,\dots
$$
We construct a subspace $V\subset {\cal{U}}_k(\g)\cap \NM$ invariant under the one-parameter subgroup $\exp(t\,X)$ with basis elements $\Gamma,\;\Gamma_1,\dots$. Due to the finite-dimensionality of the space ${\cal{U}}_k(\g)$, the subspace $V$ is also finite-dimensional. Choosing another element of the algebra $X_1\in\g$ and performing the same procedure, we obtain a $G$--invariant subspace $V$ of the space ${\cal{U}}_k(\g)\cap \NM$ containing the original identity $\Gamma$.

Now let us show that the identity $\Gamma$ belongs to the ideal $\HH$. Since $\pp\approx T_{x_0}M$, we can choose coordinates and their numbering such that
$$
\zeta_{E_{\bar{\alpha}}}(x_0)=\zeta_{\bar{\alpha}}^\alpha(x_0)\pder{}{x^\alpha}\bigg|_{x=x_0}=\pder{}{x^{\bar{\alpha}}}\bigg|_{x=x_0}\in T_{x_0}M,\quad E_{\bar{\alpha}}\in \pp.
$$
Let the identity $\Gamma$ have some order $k\geq 1$, i.e.
\[
\Gamma=\sum_{i_1\leq i_2\leq\dots \leq i_k}\Gamma^{i_1i_2\dots i_k}E_{i_1}\dots E_{i_k}+\sum_{i_1\leq i_2\leq\dots \leq i_{k-1}}\Gamma^{i_1i_2\dots i_{k-1}}E_{i_1}\dots E_{i_{k-1}}+\cdots \in {\cal{U}}_k(\g)\cap \NM.
\]
We will prove that the identity belongs to the ideal $\HH$ by induction on the order $k$. Assume that not all monomials of degree $k$ entering the identity $\Gamma$ belong to the ideal $\HH$. Then
\begin{equation}\label{e27}
\tau^M(\Gamma)\bigg|_{x=x_0}=\sum_{\dim \h<\bar{\alpha}_1\leq \bar{\alpha}_2\leq\dots \bar{\alpha}_k}\Gamma^{\bar{\alpha}_1 \bar{\alpha}_2\dots \bar{\alpha}_k}
\pder{}{x^{\bar{\alpha}_1}}\dots \pder{}{x^{\bar{\alpha}_k}}\bigg|_{x=x_0}+A_{k-1}=0.
\end{equation}
(In this formula, $A_{k-1}$ is a differential operator of order $k-1$).
The equality to zero of a differential operator is equivalent to the vanishing of all coefficients of the derivative symbols, i.e., $\Gamma^{\bar{\alpha}_1 \bar{\alpha}_2\dots \bar{\alpha}_k}=0$. In other words, there are no monomials of degree $k$ in the identity that do not belong to the ideal $\HH$. Next, taking this into account, we write out the expression for $\tau^M(\Gamma)\big|_{x=x_0}$. We obtain formula \eqref{e27} in which we must replace $k$ by $k-1$. From this follows the equality to zero of the coefficients $\Gamma^{\bar{\alpha}_1 \bar{\alpha}_2\dots \bar{\alpha}_{k-1}}$.
Continuing this process, we conclude that there are no terms entering the identity that do not belong to the ideal $\HH$. $\blacktriangleright$

Note that any $G$--invariant subspace of the right ideal $\HH$ is a two-sided ideal.

Let us prove a lemma.

\begin{lemma}
Let $\lambda\in \gf$ be a non-degenerate linear functional on the algebra $\g$, $\kk$ its polarization, $K$ a local subgroup of the group $G$ with Lie algebra $\kk$, and $Q=G/K$ a homogeneous space. Then any element of the center $\Lambda\in \Zg$ is an identity on the space $Q$, i.e., $\tau^Q(\Lambda)=0$.
\end{lemma}

\emph{Proof of Lemma 1.} $\blacktriangleleft$ Let $\tau^Q(X)=\tx=\tx^{\alpha}(q)\p/\p q^\alpha$ be the fundamental vector fields $[\tx,\theta_Y]=\theta_{[X,Y]}$. Then the functions $\tx(q,p)=\tx^{\alpha}(q)p_\alpha$ generate a representation of the Lie algebra $\g$ in the space of functions on the cotangent bundle $T^*Q$ with the canonical Poisson bracket:
\[
\{\tx(q,p),\theta_Y(q,p)\}\equiv\frac{\p\tx(q,p)}{\p p_\alpha}
\frac{\p \theta_Y(q,p)}{\p q^\alpha}- \frac{\p\theta_Y(q,p)}{\p p_\alpha}
\frac{\p \theta_X(q,p)}{\p q^\alpha}=\theta_{[X,Y]}(q,p),\quad X,Y\in\g.
\]

Each polynomial Casimir function corresponds one-to-one to an element of the center of the enveloping algebra
$$
\Lambda_k(f)=\Lambda^{i_1i_2\dots i_k}f_{i_1}f_{i_2}\dots f_{i_k}\longleftrightarrow\Lambda_k=\Lambda^{i_1i_2\dots i_k}E_{i_1}E_{i_2}\dots E_{i_k}\in\Zg
$$ (here in the enveloping algebra we use the symmetrized ordering of basis elements).

In works \cite{Do}, \cite{KP}, \cite{BS-2003}, \cite{BS-2009}, it is proved that the existence of a polarization $\kk$ of a non-degenerate linear functional $\lambda\in \gf$ allows one to introduce local canonical coordinates $(q,p)$ on the orbit $\Ol\subset\gf$ by constructing a mapping $\mu:\ T^*Q\to \Ol$ of the following form.
\begin{equation}\label{e31}
\<\mu(q,p,\lambda),X\>=\mu_X(q,p,\lambda)=\tx^{\alpha}(q)p_\alpha+\chi_X(q,\lambda),\quad
\mu_X(0,0,\lambda)=\<\lambda,X\>,\quad X\in\g.
\end{equation}
Moreover,
\begin{equation*}
\{\mu_X(q,p,\lambda),\mu_Y(q,p,\lambda)\}=\mu_{[X,Y]}(q,p,\lambda),\quad X,Y\in\g.
\end{equation*}


By definition, Casimir functions are constant on each orbit of the coadjoint representation $\Ol$. Consequently, the restriction of the Casimir function to the orbit $\Ol$ is a constant $\Lambda_k(\mu(q,p,\lambda))=\Lambda_k(\lambda)$:
\begin{equation}\label{e33}
\Lambda^{i_1i_2\dots i_k}\mu_{i_1}(q,p,\lambda)\mu_{i_2}(q,p,\lambda)\dots \mu_{i_k}(q,p,\lambda)=\Lambda^{i_1i_2\dots i_k}\lambda_{i_1}\lambda_{i_2}\dots \lambda_{i_k}.
\end{equation}
Equation \eqref{e33} must hold identically for arbitrary values of the variables $p,q$ from an open neighborhood of zero. Taking into account definition \eqref{e31}, for the highest powers of the polynomial in the variables $p_\alpha$ we obtain the equalities
\begin{equation}\label{e34}
\Lambda^{i_1i_2\dots i_k}\theta_{i_1}^{\alpha_1}(q)\theta_{i_2}^{\alpha_2}(q)\dots \theta_{i_k}^{\alpha_k}(q)=0,
\end{equation}
which must hold for arbitrary values of the variables $q^\alpha$.

Taking into account equality \eqref{e34}, we have:
\begin{multline*}
\tau^Q(\Lambda_k)=\Lambda^{i_1i_2\dots i_k}\left(\theta_{i_1}^{\alpha_1}(q)\frac{\p}{\p q^{\alpha_1}}\right)\left(\theta_{i_2}^{\alpha_2}(q)\frac{\p}{\p q^{\alpha_2}}\right)\dots \left(\theta_{i_k}^{\alpha_k}(q)\frac{\p}{\p q^{\alpha_k}}\right)=\\
=\left(\Lambda^{i_1i_2\dots i_k}\theta_{i_1}^{\alpha_1}(q)\theta_{i_2}^{\alpha_2}(q)\dots \theta_{i_k}^{\alpha_k}(q)\right) \frac{\p}{\p q^{\alpha_1}}\frac{\p}{\p q^{\alpha_2}}\dots \frac{\p}{\p q^{\alpha_k}}+A_{k-1}=A_{k-1}.
\end{multline*}
Here $A_{k-1}$ is an invariant differential operator of order $k-1$.
It can be verified that for the polarization $\kk$ and the homogeneous space $Q$, condition \eqref{e22} holds. According to Remark 2, this means that the operator $A_{k-1}\in \mathfrak{D}(Q)$ is a Casimir operator: $A_{k-1}=\tau^Q(\Lambda_{k-1})$ (or both sides of this equality are zero). Continuing this reasoning, we obtain the chain of equalities:
\[
\tau^Q(\Lambda_{k})=\tau^Q(\Lambda_{k-1})=\dots=\tau^Q(\Lambda_{1})=0.
\]

Let us comment on the last equality in this chain. Equation \eqref{e34} for $k=1$ has the form: $\Lambda^i\theta_i^\alpha(q)=0$. Hence
$
\tau^Q(\Lambda_1)=\Lambda^i\theta_i^\alpha(q){\p}/{\p q^\alpha}=0.
$
$\blacktriangleright$

\section{Harmonic Variables for Laplace Operators on Lie Groups}

Suppose there exists an invariant bilinear symmetric non-degenerate form on the algebra $\g$:
\[
B(X,Y)=B(Y,X)=B(\Ad_g X,\Ad_g Y)\in\R,\quad X,Y\in\g.
\]
This form defines a two-sided invariant Riemannian metric on the Lie group $G$:
\begin{gather}\label{e28}
\textrm{ g}=B_{ij}\omega^i(g_x)\omega^j(g_x)=
B_{ij}\sigma^i(g_x)\sigma^j(g_x)=g_{kl}(x)dx^kdx^l,\\ \nonumber
g_{kl}(x)=B_{ij}\omega^i_k(x)\omega^j_l(x)=B_{ij}\sigma^i_k(x)\sigma^j_l(x),\quad B_{ij}=B(E_i,E_j),\quad \det B_{ij}\neq 0.
\end{gather}
Any semi-simple group possesses an invariant form $B$ (the Killing tensor). However, there also exist non-semi-simple groups admitting this form, for example, some solvable groups. Note that any group with such an invariant form is unimodular.

The non-degenerate invariant form $B$ defines an element of the center of the enveloping algebra:
\[
{\Lambda}=B^{ij}E_iE_j\in {\Zg},\quad B^{ij}= (B_{ij})^{-1}.
\]
The Casimir operator
$\tau^R({\Lambda})=\tau^L({\Lambda})=B^{ij}\xi_i(g)\xi_j(g)=B^{ij}\eta_i(g)\eta_j(g)$
coincides with the Laplace operator on the Lie group with the metric \eqref{e28}:
\begin{equation}\label{e29}
\Delta_2(G)=B^{ij}\xi_i(g)\xi_j(g).
\end{equation}
The last equality is due to the unimodularity property of the group. For a non-unimodular group, the operator $C^{ij}\xi_i(g)\xi_j(g)$ differs from the Laplace operator with the left-invariant metric $\textrm{g}=C_{ij}\omega^i(g)\omega^j(g)$ by a first-order differential operator.

\begin{theorem}
For the Laplace operator \eqref{e29} on the Lie group $G$, there exists a complex of harmonic variables consisting of $(1/2)(\dim\g-\ind\g)$ generators.
\end{theorem}


\emph{Proof of Theorem 2.} $\blacktriangleleft$
Since the scheme for constructing harmonic variables is also valid in the more general case of homogeneous spaces, we will consider it here.

Let the Laplace operator on a homogeneous space $M=G/H$ have the form:
\begin{equation}\label{e30}
\Delta_2(M)=\tau^M(B^{ij}E_iE_j)=B^{ij}\zeta_i(x)\zeta_j(x),\quad \Lambda=B^{ij}E_iE_j\in \Zg.
\end{equation}
In the next section, we will discuss the conditions under which formula \eqref{e30} holds. If we consider a Lie group as a homogeneous space on which the group acts on itself by right translations, then formula \eqref{e30} in this case coincides with formula \eqref{e29}.

The algorithm for constructing harmonic variables for the Laplace operator \eqref{e30} consists in sequentially solving two problems.

\emph{Problem 1.} For the central element ${\Lambda}=B^{ij}E_iE_j\in\Zg$, choose a subalgebra $\kk$ such that this element ${\Lambda}$ is an identity on the homogeneous space $Q=G/K$, where $K$ is a local Lie group with algebra $\kk$, i.e., $\tau^Q({\Lambda})=0$.

\emph{Problem 2.} Construct in coordinates from equations \eqref{e20} the morphism $I:\ M\to Q$. The functions $q=I(x)$ will constitute a complex of harmonic variables with $\dim Q$ generators.

Let us explain the last statement. Let $q=I(x)$ be a morphism, i.e., equality \eqref{e20} holds. Then for an arbitrary function $v(q)$ we have:
\begin{multline*}
\Delta_2(M)v(q(x))=B^{ij}\zeta_i(x)\zeta_j(x)v(q(x))=B^{ij}\zeta_i(x)v'(q)\zeta_j(x)q(x)=B^{ij}\zeta_i(x)\theta_j(q)v(q)=\\
=B^{ij}\theta_i(q)\theta_j(q)v(q)=\tau^Q({\Lambda})v(q)=0,\quad q=I(x).
\end{multline*}


In the monograph \cite{Diksime}, it is stated that for a non-degenerate linear functional $\lambda\in\gf$, its polarization $\kk$ always exists, which, according to Lemma 1, is exactly the sought subalgebra. Questions of constructing a polarization are discussed, in particular, in \cite{K-S}.

Let us proceed to the solution of Problem 2. In the case of the morphism $I:\ G\to Q$, Proposition 1 imposes no restrictions on the subgroup $K\subset G$, i.e., such a morphism always exists, which completes the proof of the theorem. $\blacktriangleright$

Let us present an algebraic method for constructing complexes of harmonic variables on Lie groups for the Laplace operators \eqref{e29}.

Equations \eqref{e21} for a fixed basis $\{E_i\}$ of the algebra $\g$ in coordinates have the form:
\begin{equation}\label{e311}
\xi_i^j(a)\frac{\p q^\alpha}{\p a^j}=\theta_i^\alpha(q),\quad i=1,\dots,n=\dim\g, \quad \alpha=1,\dots,m=\dim Q=(\dim\g-\ind\g)/2.
\end{equation}
Here $a=(a^1,\dots,a^n)$ are the coordinates of the group element $g_a$ from an open neighborhood of the identity $U_e\subset G$.

Let $x\in Q$ be a point in general position. Differentiate equality \eqref{e15} with respect to the variables $b^j$ at zero and denote $q^\alpha=\psi^\alpha(x,a)$. Then, taking into account the definitions of the fundamental vector fields \eqref{e8}, \eqref{e16}, we obtain system \eqref{e311}. In other words, the action functions of the Lie group on the homogeneous space give the desired morphism. To find it, we use formula \eqref{e16a}, which in our case has the form:
\begin{equation*}
\theta_i^\beta(x)\frac{\p q^\alpha}{\p x^\beta}=
(\Ad_{g_a^{-1}})_i^k\theta_k^\alpha(q)
\end{equation*}
Let us rewrite this system of equations in matrix form:
\begin{equation}\label{e33a}
\begin{pmatrix}
\theta^1_1(x)&\dots &\theta^m_1(x)\\
\theta^1_2(x)&\dots &\theta^m_2(x)\\
\dots&\dots &\dots\\
\theta^1_n(x)&\dots &\theta^m_n(x)\\
\end{pmatrix}
\begin{pmatrix}
r_1^{(\alpha)}\\
\dots\\
r_m^{(\alpha)}
\end{pmatrix} =
\begin{pmatrix}
\tilde{\theta}_1^{(\alpha)}(a,q)\\
\tilde{\theta}_2^{(\alpha)}(a,q)\\
\dots\\
\tilde{\theta}_n^{(\alpha)}(a,q)\\
\end{pmatrix}.
\end{equation}
Here $r_\beta^{(\alpha)}=\p q^\alpha/\p x^\beta,\ \tilde{\theta}_i^{(\alpha)}(a,q)=(\Ad_{g_a^{-1}})_i^k\theta_k^\alpha(q)$.
We will consider formula \eqref{e33a} as a system of $n$ linear inhomogeneous algebraic equations for $m<n$ unknowns $r_\beta^{(\alpha)}$ for each fixed index $\alpha$. It is well known that such a system is consistent if and only if the rank of the coefficient matrix of the unknowns equals the rank of the augmented matrix. This is equivalent to the vanishing of all minors $M_i^{(\alpha)}$ of size $(m+1)$ (we will call these equations for the unknown quantities $q$ \textit{algebraic}):
\begin{equation}\label{e34a}
M_i^{(\alpha)}(x,a,q)=0,\quad i=1,\dots,C_n^{m+1}
\end{equation}
of all augmented matrices $L^{(\alpha)}(x,a,q),\ \alpha=1,\dots,m=\dim Q$:
\begin{equation*}
L^{(\alpha)}(x,a,q)=
\begin{pmatrix}
\theta^1_1(x)&\dots &\theta^m_1(x)&\tilde{\theta}_1^{(\alpha)}(a,q)\\
\theta^1_2(x)&\dots &\theta^m_2(x)&\tilde{\theta}_2^{(\alpha)}(a,q)\\
\dots&\dots &\dots&\dots\\
\theta^1_n(x)&\dots &\theta^m_n(x)&\tilde{\theta}_n^{(\alpha)}(a,q)\\
\end{pmatrix}.
\end{equation*}

Let us discuss the question: can all quantities $q$ be determined from the algebraic equations \eqref{e34a}?

Assume the contrary, i.e., not all quantities $q$ can be found from equations \eqref{e34a}. This means that there exists a coordinate transformation $q^\beta$ after which at least one of the coordinates, say $q^1$, does not enter the coefficients of the vector fields $\theta^\alpha_X(q)=$ \\$=\theta^\alpha_X(q^2,q^3,\dots,q^m)$. In turn, this is equivalent to the existence of an invariant vector field $\upsilon=\p/\p q^1$: $[\theta_X,\upsilon]=0$. However, due to the fulfillment of condition \eqref{e22} (commutativity of the space $Q$), every invariant operator is a Casimir operator. According to Lemma 1, the Casimir operator in the space $Q$ is zero, and thus our assumption is false.

Note the following. System of equations \eqref{e311} is invariant under left translations, i.e., if $q=q(a)$ is some solution, then $\tilde{q}(a)= T^L_{g_b^{-1}}q(a)=q(\varphi(b,a))$ is also a solution. From equality \eqref{e15} it follows that if $q(a)=\psi(x_0,a)$, then $\tilde{q}(a)=\psi(x,a)$, where $x=\psi(x_0,b)$. Thus, to obtain the general solution of system \eqref{e311}, we can choose two paths. The first path: in the algebraic equations \eqref{e34a}, consider the point $x\in Q$ as a point in general position with arbitrary coordinates. The second path: in equations \eqref{e34a}, set $x=x_0$ (for example, $x_0=0$), which significantly simplifies the solution of these equations, and then, in the solution $q(a)$, make the substitution $a\to \varphi(b,a)$.

\emph{Example 2} (Harmonic variable on the group $SO(3)$).
$\blacktriangleleft$ The Lie algebra $so(3)$ of the group $SO(3)$ has the following non-zero basis commutation relations:
\[
[E_1,E_2]=E_3,\quad [E_2,E_3]=E_1,\quad [E_3,E_1]=E_2.
\]
Let us choose local canonical coordinates of the second kind $(\alpha,\beta,\gamma)$ in an open neighborhood of the identity $U_e\subset SO(3)$
\[
g_{\alpha,\beta,\gamma}=\exp({\alpha E_1})\exp({\beta E_2})\exp({\gamma E_3})\in U_e.
\]
In these coordinates, the left-invariant vector fields have the form ($\xi_i=\tau^R(E_i)$):
\begin{eqnarray*}
\xi_1&=&\frac{\cos\gamma}{\cos \beta}\frac{\p}{\p \alpha}+\sin \gamma \pder{}{\beta}-  \cos\gamma\tgg\beta\pder{}{\gamma},\\
\xi_2&=&- \frac{\sin\gamma}{\cos\beta}\pder{}{\alpha} +\cos\gamma\pder{}{\beta}+
\sin\gamma\tgg\beta\pder{}{\gamma},\\
\xi_3&=& \pder{}{\gamma}
\end{eqnarray*}
Let us also write out the inverse matrix of the adjoint representation in these coordinates.
\begin{equation*}
Ad_{g^{-1}_{\alpha,\beta,\gamma}}=
\begin{pmatrix}
\cos\beta\cos\gamma&  \sin\alpha \sin\beta\cos\gamma + \cos\alpha\sin\gamma &
-\cos\alpha \sin\beta\cos\gamma + \sin\alpha\sin\gamma\\
-\cos\beta\sin\gamma&
-\sin\alpha \sin\beta\sin\gamma + \cos\alpha\cos\gamma &
\cos\alpha \sin\beta\sin\gamma + \sin\alpha\cos\gamma
  \\
\sin\beta& -\sin\alpha\cos\beta & \cos\alpha\cos\beta
\end{pmatrix}
\end{equation*}

The Killing tensor in our basis has the form of the identity matrix; correspondingly, the element $\Lambda=E_1^2+E_2^2+E_3^2$ generates the center of the enveloping algebra.
According to formula \eqref{e29}, we have the Laplace operator on the group $SO(3)$ with a bi-invariant metric.
\[
\DD(SO(3))=\xi_1^2+\xi_2^2+\xi_3^2=\frac{1}{\cos^2\beta}\left (\pder{^2}{\alpha^2}+\pder{^2}{\gamma^2}-2\sin\beta\frac{\p}{\p \alpha}\frac{\p}{\p \gamma} \right)+\pder{^2}{\beta^2}-\tgg\beta\pder{}{\beta}.
\]

The algebra $so(3)$ is three-dimensional and has index $\ind so(3)=1$. By definition \eqref{e9}, the polarization $\kk$ is two-dimensional; consequently, according to Theorem 2, we have a complex consisting of one harmonic variable.

Let us fix a non-degenerate linear functional $\lambda=(\lambda_1,0,0)$ and construct its polarization $\kk\subset\g^c$: $\kk=\<E_1,E_2+i E_3\>_\CC$ (here $\<v_1,v_2,\dots\>_{\F}$ is the linear span of vectors $v_1,v_2,\dots$ over the field $\F$).

In the space of smooth functions on the set of right cosets $Q=G^c/\exp(\kk)$, the regular representation $\tau^Q$ of the algebra $\g^c$ acts:
\[
\tau^Q(E_1)=\theta_1(q)=-i\sin q\frac{\p}{\p q},\quad
\tau^Q(E_2)=\theta_2(q)=-i\cos q\frac{\p}{\p q},\quad
\tau^Q(E_3)=\theta_3(q)=\frac{\p}{\p q}.
\]

Let us illustrate Theorem 1 with this example. The central element $\Lambda$ is represented as:
\[
\Lambda=E_1^2+E_2^2+E_3^2=Y_1 (E_1+i)+Y_2 (E_2-iE_3),\quad
 Y_1=E_1\in\kk,\quad Y_2=E_2+i E_3\in\kk.
\]
Thus, this element belongs to the set $\HH\cap\Zg$ and, according to Theorem 1, is an identity (which can be easily verified directly):
\[
\tau^Q(\Lambda)=\theta_1^2(q)+\theta_2^2(q)+\theta_3^2(q)=0.
\]

Using the matrix $Ad_{g^{-1}_{\alpha,\beta,\gamma}}$ and the form of the fields $\xi_i,\ \theta_i$, we construct the matrix $L(x,a,q)$, which has three rows and two columns:
\[
L(x,a,q)=\left(
\begin{array}{cc}
 -i \sin (x) & \sin (\beta )-i \cos (\beta ) \sin (q-\gamma ) \\
 -i \cos (x) & -i (\cos (\alpha ) \cos (q-\gamma )+\sin (\alpha ) (\sin (\beta ) \sin (q-\gamma )-i \cos (\beta ))) \\
 1 & \cos (\alpha ) (\cos (\beta )+i \sin (\beta ) \sin (q-\gamma ))-i \cos (q-\gamma ) \sin (\alpha ) \\
\end{array}
\right)
\]

Equating the three second-order minors of this matrix to zero, we obtain three equations \eqref{e34a} for the single unknown quantity $q$.
These equations can be considered as a system of three linear equations for the two unknown quantities $\cos q,\ \sin q$. From the first two equations we find:
$\cos q=F_1(\alpha,\beta,\gamma,x),\ \sin q=F_2(\alpha,\beta,\gamma,x)$. Next, we substitute these quantities into the third equation and verify that it holds identically. Furthermore, it is easy to verify the identity: $F_1^2+F_2^2=1$.

In the Introduction, we noted that any arbitrary function of a harmonic variable is itself a harmonic variable. Therefore, to avoid dealing with inverse trigonometric functions, we present as the answer an equivalent harmonic variable $F_1(\alpha,\beta,\gamma,x)$. By substitution into the differential equations \eqref{e4a}, \eqref{e4b}, it is easy to verify that for an arbitrary complex value of the parameter $x$, the function
\[
q=F_1(\alpha,\beta,\gamma,x)=\frac{\cos\gamma(\cos x\cos\alpha+ i\sin\alpha)-\sin\gamma(\cos\beta\sin x+\sin\beta(\sin\alpha\cos x-i\cos\alpha))}{\cos\beta (\cos\alpha+i\sin\alpha\cos x)-i\sin\beta\sin x}
\]
is a harmonic variable.$\blacktriangleright$

\section{Harmonic Variables on Homogeneous Spaces}

Let $M=G/H$ be a homogeneous space on which the transformation group $G$ acts effectively, and let $\Lambda_2=B^{ij}E_iE_j$ be an element of the center of the enveloping algebra $U(\g)$, with $B^{ij}=B^{ji}$, $\rank B^{ij}\geq \dim M$. Consider the invariant second-order differential operator $H$.
\[
H=\tau^M(\Lambda_2)=B^{ij}\zeta_i(x)\zeta_j(x),\quad [\zeta_X,H]=0,\quad X\in\g.
\]

We will assume that the coefficients of the second derivatives in this operator are the contravariant components of a metric tensor: $g^{ab}(x)=B^{ij}\zeta_i^a(x)\zeta_j^b(x)$. Let us construct the Laplace operator for this metric:
\[
\DD=\frac{1}{\sqrt{|g|}}\pder{}{x^a}\sqrt{|g|}g^{ab}(x)\pder{}{x^b},\quad
g_{ab}=||g^{ab}||^{-1},\quad g=\det g_{ab}.
\]

The vectors $\zxx$ are Killing vectors for the metric tensor $\textrm{g}=g_{ab}(x)d x^a dx^b$, and consequently $[\zxx,\DD]=0$. Thus, the two invariant operators $H$ and $\DD$ may differ by an invariant first-order differential operator $\zeta_0=\zeta_0^a(x)\p/\p x^a$:
\[
\DD=H+\zeta_0,\quad [\zxx,\zeta_0]=0.
\]
In this paper, we will consider the class of homogeneous spaces for which every invariant first-order differential operator is a Casimir operator, i.e., there exists $\Lambda_0\in \z$ such that $\zeta_0=\tau^M(\Lambda_0)$ (here and below, $\z$ is the center of the algebra $\g$). We will call such homogeneous spaces \textit{special} or, for short, \textit{s-spaces}.
In other words, the Lie algebra $\sm$ of the automorphism group of the homogeneous s-space $M$ consists only of the center $\z$ of the algebra $\g$. A special case of s-spaces are commutative spaces. The latter, in turn, generalize symmetric homogeneous spaces.

For an s-space $M$, the following formula holds:
\begin{equation}\label{e35a}
\DD(M)=\tau^M(\Lambda),\quad \Lambda=B^{ij}E_iE_j+E_0\in{\Zg},\quad E_0\in\z\subset\Zg.
\end{equation}

The criterion for the space $M=G/H$ to belong to the class of s-spaces is the following relation:
\begin{equation}\label{e35}
\sm=\g^{\h}/\h=\z,\quad \g^{\h}=\{X\in\g\mid [X,\h]\subset\h \}.
\end{equation}

Let us explain the last formula. By definition, the isotropy subalgebra $\h$ is an ideal in the algebra $\g^{\h}$, as a result of which $\g^{\h}/\h$ is a well-defined factor-algebra. By construction, the center $\z\subset \g^{\h}$, but due to the effectiveness of the action of the group $G$ on the space $M$, $\h\cap\z=0$, i.e., the factor-algebra always contains the center $\z$ "entirely". In this sense, criterion \eqref{e35} is equivalent to the minimality condition of the factor-algebra $\g^{\h}/\h$, i.e., $\g^\h=\h\oplus \z$.

In this section, we will discuss the construction of harmonic variables for the Laplace equations \eqref{e35a} on homogeneous s-spaces.

By analogy with Theorem 2, we could formulate here a statement about the existence for homogeneous s-spaces of a complex consisting of $(\dim\g-\dim\g^\lambda)/2$ independent harmonic variables, where $\lambda\in\h^{\bot}$. However, for a degenerate linear functional $\lambda$, a polarization does not always exist, and therefore it is not yet possible to prove this statement.
To solve Problem 1 --- constructing the subalgebra $\kk$ and the space $Q\simeq G/K$, we will use Theorem 1.

Let us proceed to the discussion of the algorithm for solving Problem 2 --- constructing the morphism $I: M\to Q\simeq G/K$. Let us represent the system of equations \eqref{e21} in matrix form.
\begin{equation}\label{e36}
\begin{pmatrix}
\zeta^1_1(x)&\dots &\zeta^m_1(x)\\
\zeta^1_2(x)&\dots &\zeta^m_2(x)\\
\dots&\dots &\dots\\
\zeta^1_n(x)&\dots &\zeta^m_n(x)\\
\end{pmatrix}
\begin{pmatrix}
r_1^{(\bar\alpha)}\\
\dots\\
r_m^{(\bar\alpha)}
\end{pmatrix} =
\begin{pmatrix}
{\theta}_1^{\bar\alpha}(q)\\
{\theta}_2^{\bar\alpha}(q)\\
\dots\\
{\theta}_n^{\bar\alpha}(q)\\
\end{pmatrix},\quad \bar\alpha=1,\dots,l=\dim Q.
\end{equation}
Here $r_\beta^{(\bar\alpha)}=\p q^{\bar\alpha}/\p x^\beta$.

If we consider equations \eqref{e36} for each fixed index $\bar\alpha$ as a system of linear algebraic equations for the unknown quantities $r_\beta^{(\bar\alpha)}$, then a necessary and sufficient condition for the existence of solutions is the vanishing of all minors $M_i^{(\bar\alpha)}$ of order $m+1$
\begin{equation}\label{e37}
M_i^{(\bar\alpha)}(x,q)=0,\quad i=1,\dots, C_n^{m+1},\quad \bar\alpha=1,\dots,l=\dim Q;
\end{equation}
of the augmented matrix:
\begin{equation}\label{e38}
L^{(\bar\alpha)}(x,q)=
\begin{pmatrix}
\zeta^1_1(x)&\dots &\zeta^m_1(x)&{\theta}_1^{(\bar\alpha)}(q)\\
\zeta^1_2(x)&\dots &\zeta^m_2(x)&{\theta}_2^{(\bar\alpha)}(q)\\
\dots&\dots &\dots&\dots\\
\zeta^1_n(x)&\dots &\zeta^m_n(x)&{\theta}_n^{(\bar\alpha)}(q)\\
\end{pmatrix}.
\end{equation}
As before, we will call equations \eqref{e37} for the unknown quantities $q=I(x)$ algebraic equations.

It is easy to see that all unknown quantities $q(x)$ can be found from the algebraic equations only in the case of a zero automorphism algebra: $\s(Q)=0$. In the general case, with a special choice of coordinates $q^{\bar{\alpha}}$, system \eqref{e21} takes a triangular form, and thus the morphism $q=I(x)$ can also be found by quadratures.

Let us illustrate the above for a one-dimensional automorphism algebra ($\dim \s(Q)=1$). In this case, there exists one invariant vector field
$$
\tau(q)=\tau^{\ba}(q)\frac{\p}{\p q^{\ba}}:\quad [\tau(q),\theta_X(q)]=0,\quad X\in\g.
$$
Let us change the variables $q$ so that $\tau(q)=\p/\p q^1$. From the commutativity condition follows the independence of the coefficients $\theta^{\ba}_i$ from the variable $q^1$: $\theta^{\ba}_i(q)=\theta^{\ba}_i(q^2,\dots,q^{l})$.
Thus, from the algebraic equations we can find all components $q^{\ba}=q^{\ba}(x)$ except for the quantity $q^1$. The latter is found by integrating the first equation of system \eqref{e21}:
\begin{equation*}
\zxx^\beta(x)\frac{\p q^{1}}{\p x^\beta}=\tx^{1}(q^2(x),\dots,q^l(x)),\quad  X\in\g.
\end{equation*}

\emph{Example 3} (Harmonic variable on the two-dimensional sphere). $\blacktriangleleft$ For methodological purposes, let us consider an example whose result we already presented in the Introduction.

We will represent the two-dimensional sphere as the space of right cosets: $S^2\approx SO(3)/SO(2)$. In the algebra $\g=so(3)$, we choose a basis $\{L_{ij}\} \  (1\leq i< j \leq 3)$ with commutation relations:
\[
[L_{12},L_{13}]=-L_{23},\quad [L_{12},L_{23}]=L_{13},\quad [L_{13},L_{23}]=-L_{12}.
\]
The Lie algebra of the isotropy subgroup $H=SO(2)$ is the algebra $\h=so(2)=\< L_{23}\>_{\R}$.

The generators of the transformation group $G=SO(3)$ on the homogeneous space $M=S^2$ have the standard form ($\zeta_{ij}=\tau^M(L_{ij})$):
\[
\zeta_{12}=\pder{}{\varphi},\quad
\zeta_{13}= -\sin\varphi \cotg \theta\pder{}{\varphi}+ \cos\varphi\pder{}{\theta},\quad
\zeta_{23}= \cos\varphi \cotg \theta\pder{}{\varphi}+ \sin\varphi\pder{}{\theta}.
\]
The Laplace operator on the sphere is a Casimir operator and has the form \eqref{e35a} with $\Lambda_0=0$:
\[
\DD(S^2)=\tau^M(\Lambda)=\zeta_{12}^2+\zeta_{13}^2+\zeta_{23}^2=\left( \frac{1}{\sin \theta}\frac{\p}{\p \theta}\sin \theta \frac{\p}{\p \theta}+
 \frac{1}{\sin^2\theta}\frac{\p^2}{\p \varphi^2}\right).
\]
Here $\Lambda=L_{12}^2+L_{13}^2+L_{23}^2\in \Zg$.

Let us choose a subalgebra $\kk\in\g^c$ such that the element $\Lambda$ belongs to the ideal $\HH$ from Theorem 1: $\kk=\<L_{23},L_{12}+iL_{13}\>_C$. It is easy to compute the generators of the (local) action of the group on the space of right cosets $Q=G^c/\exp(\kk)$.
\[
\theta_{12}=\tau^Q(L_{12})=\pder{}{q},\quad \theta_{13}=\tau^Q(L_{13})=i \cos q \pder{}{q},\quad \theta_{23}=\tau^Q(L_{23})=i\sin q\pder{}{q}.
\]

Let us compose the matrix $L(x,q)$ and compute its determinant (the only minor of order 3):
\[
L=\begin{pmatrix}
1 & 0 & 1 \\
-\ctg \theta\sin \varphi & \cos\varphi & i\cos q \\
\ctg\theta \cos\varphi & \sin\varphi & i\sin q
\end{pmatrix},\quad \det L=-\ctg\theta-i\sin (\varphi-q).
\]
Equating the determinant to zero, we obtain the expression for the harmonic variable \\ $q(\varphi,\theta)=\varphi -  i\arcsinh ( \cotg \theta )$. $\blacktriangleright$


It should be noted that in all the examples considered above, the harmonic variables were expressed in elementary functions. However, in the general case, it is not always possible to solve equations \eqref{e37} explicitly for the unknown quantities $q$ (for instance, this could be a polynomial equation of degree higher than four in the variable $q$ with coefficients depending on $x$). In this case, one has to state that the independent equations of system \eqref{e37} define the harmonic variables $q=q(x)$ implicitly. This is especially characteristic for the case of a multidimensional space $Q$, when we find complexes of harmonic coordinates. Let us provide an example of finding a complex in explicit form.

\emph{Example 4} (Complex of harmonic variables in the symmetric space $SO(2,2)/SO(1,1)\times SO(1,1)$). $\blacktriangleleft$
Let us present the commutation relations of the algebra $\g=so(2,2)=\<E_{ij}\>_{\R}$.
\[
[E_{ij},E_{kl}]=\textrm{g}_{kj}E_{il}- \textrm{g}_{ki}E_{jl} + \textrm{g}_{li}E_{jk} - \textrm{g}_{lj}E_{ik},\quad E_{ij}=-E_{ji},\quad 1\leq i,j,k,l \leq 4.
\]
Here $\textrm{g}=\diag(1,-1,-1,1)$ is a diagonal matrix.

The Lie algebra $\h=so(1,1)\oplus so(1,1)$ of the isotropy subgroup
in our case has the form $\h=\<E_{12},E_{34}\>_\R$. Let us write out the fundamental vector fields in local coordinates on the homogeneous symmetric space $M=SO(2,2)/SO(1,1)\times SO(1,1)$.
\begin{eqnarray*}
\zeta_{12}&=& x_2 \pder{}{x_1}+x_1 \pder{}{x_2} + x_4 \pder{}{x_3}+ x_3 \pder{}{x_4},\\
\zeta_{13}&=&(\frac12+ x_2x_3-x_1x_4)\pder{}{x_1}+
( x_1x_3-x_2x_4)\pder{}{x_2}  +
x_3x_4\pder{}{x_3}+
\frac12(x_3^2+x_4^2-1)\pder{}{x_4},\\
\zeta_{14}&=&(\frac12+ x_2x_3-x_1x_4)\pder{}{x_1}+
( x_1x_3-x_2x_4)\pder{}{x_2}  +
x_3x_4\pder{}{x_3}+
\frac12(x_3^2+x_4^2+1)\pder{}{x_4},\\
\zeta_{23}&=&( x_1x_3-x_2x_4)\pder{}{x_1}+
(\frac12+ x_2x_3-x_1x_4)\pder{}{x_2}
-\frac12(1+x_3^2+x_4^2)\pder{}{x_3}-
x_3x_4\pder{}{x_4},\\
\zeta_{24}&=&( x_1x_3-x_2x_4)\pder{}{x_1}+
(\frac12+ x_2x_3-x_1x_4)\pder{}{x_2}
+\frac12(1-x_3^2-x_4^2)\pder{}{x_3}-
x_3x_4\pder{}{x_4},\\
\zeta_{34}&=& -x_1 \pder{}{x_1}-x_2 \pder{}{x_2} + x_3 \pder{}{x_3}+ x_4 \pder{}{x_4}.
\end{eqnarray*}
Here $\zeta_{ij}=\tau^M(E_{ij})$.

Let us construct a general quadratic element of the center of the algebra $\Ug$:
\[
\Lambda=A  (-E_{12}E_{34}-E_{23}E_{14}+E_{13}E_{24})+
B(E_{12}^2+E_{13}^2-E_{14}^2-E_{23}^2+E_{24}^2+E_{34}^2).
\]
$A,\, B$ are real constants, $A\neq \pm 2 B$.

The Laplace operator has the form \eqref{e35a} (with $\Lambda_0=0$):
\begin{multline*}
\DD(M)=\tau^M(\Lambda)=\left(B(x_1^2+x_2^2)+A x_1x_2\right)\left(\frac{\p^2}{\p x_1^2}+\frac{\p^2}{\p x_2^2}\right)+
\left(A(x_1^2+x_2^2)+4B x_1x_2\right)\frac{\p^2}{\p x_1 \p x_2}+\\
+\frac{A}{2}\frac{\p^2}{\p x_1 \p x_3}-B\frac{\p^2}{\p x_1\p x_4}+
B\frac{\p^2 }{\p x_2\p x_3}-\frac{A}{2}\frac{\p^2}{\p x_2\p x_4}+ 2(2Bx_1+ A x_2 )\pder{}{x_1}+2( A x_1+2 Bx_2)\pder{}{x_2}.
\end{multline*}

Guided by Theorem 1, we choose a subalgebra $\kk$ containing the subalgebra $\h$ such that in the space of cosets $Q=G/\exp(\kk)$, the central element $\Lambda$ is an identity, i.e., $\tau^Q(\Lambda)=0$:
$$
\kk=\<E_{34},E_{12},-E_{13}+E_{14}-E_{23}+E_{24},
E_{13}+E_{14}+E_{23}+E_{24}\>_\R
$$
Now let us write out the fundamental vector fields $\theta_{ij}(q)=\tau^Q(E_{ij})$:
\begin{eqnarray*}
\theta_{12}=(2 q_1^2 q_2-q_1)\pder{}{q_1}-q_2\pder{}{q_2},\quad
\theta_{13}=(q_1 q_2-\frac12)\pder{}{q_1}+\frac12(1-q_2^2)\pder{}{q_2},\\
\theta_{14}=(\frac12-q_1 q_2)\pder{}{q_1}+\frac12(1+q_2^2)\pder{}{q_2},\quad
\theta_{23}=(\frac12+q_1^2-q_1q_2+q_1^2q_2^2)\pder{}{q_1}-
\frac12(1+q_2^2)\pder{}{q_2},\\
\theta_{24}=(-\frac12+q_1^2+q_1q_2-q_1^2q_2^2)\pder{}{q_1}+
\frac12(-1+q_2^2)\pder{}{q_2},\quad
\theta_{34}=q_1\pder{}{q_1}-q_2\pder{}{q_2}.
\end{eqnarray*}

Using the coefficients of the vector fields $\zeta_{ij}^{\alpha}(x),\ \theta_{ij}^{\bar{\alpha}}(q)$, we construct according to formula \eqref{e38} two matrices $L^{(1)}(x,q),\ L^{(2)}(x,q)$, each having six rows and five columns. Next, we equate to zero all 12 fifth-order minors of these matrices. We thus obtain for this case the algebraic equations \eqref{e37}. Let us present all solutions of the obtained algebraic equations for the unknowns $q_1,\; q_2$.
\begin{eqnarray*}
1.\ \{q_1(x),q_2(x)\}&=&\{x_3,x_4\};\\
2.\ \{q_1(x),q_2(x)\}&=&
\left\{\frac{ 1 + (x_1 + x_2) ( x_3 - x_4)}{1 + 2 x_3 (x_1 + x_2)},
x_3+x_4\right\};\\
3.\ \{q_1(x),q_2(x)\}&=&
\left\{\frac{ (x_1 - x_2) (x_3 - x_4)}{-1 + 2 (x_1 - x_2) x_3},
\frac{x_1 - x_2}{1 - (x_1 - x_2) (x_3 + x_4)}
\right\};\\
4.\ \{q_1(x),q_2(x)\}&=&
\left\{
\frac{1 + 2 (x_2 x_3 -
     x_1 x_4) - (x_1^2 - x_2^2) (x_3^2   - x_4^2)}{x_2 -
    x_3 (x_1^2 - x_2^2)},
\frac{1 - (x_1 - x_2) ( x_3 + x_4)}{x_1 - x_2}
    \right\}.
\end{eqnarray*}
(When writing out the complexes of harmonic variables, we applied transformations for simplification: $q_{\bar{\alpha}}\to \Phi_{\bar{\alpha}}(q_1,q_2)$). Note that by acting on the obtained complexes with the transformation group $SO(2,2)$ with generators $\zeta_{ij}^{\alpha}(x)$, we will obtain complexes depending on 6 parameters. $\blacktriangleright$

\section*{Conclusion}

In this paper, we introduced the definition of harmonic variables and complexes. For a certain class of Riemannian spaces, we presented an algebraic algorithm for finding them. However, there are still many unsolved important problems in this area. In particular, the author does not yet know the answer to the question: does every multidimensional Riemannian or pseudo-Riemannian space admit harmonic variables? What is the maximal complex for a given homogeneous space? For example, above we proved Theorem 2 on the existence of a complex for Lie groups with a bi-invariant metric, consisting of $(\dim\g-\ind\g)/2$ harmonic variables, but this does not imply the non-existence of a larger complex. These and other questions are the subject of further research.

In our opinion, the problem of constructing harmonic variables on St\"ackel spaces of various types (spaces in which the Hamilton-Jacobi equation and the Klein-Gordon-Fock equation are solved by the method of separation of variables \cite{B-O}) is also very promising.

The author would like to thank Prof. A.A. Magazev for useful discussions of the results of this work.

\end{document}